\documentclass[11pt,a4paper,reqno]{amsart}

\usepackage[lmargin=1.2in, rmargin= 1.2in, tmargin=1.5in]{geometry}

\title{Discrete vertices in simplicial objects internal to a monoidal category}

\author{Arne Mertens}
\address[Arne Mertens]{Universiteit Antwerpen, Departement Wiskunde, Middelheimcampus,
Middelheimlaan 1,
2020 Antwerp, Belgium}
\email{arne.mertens@uantwerpen.be}

\thanks{
This project has received funding from the European Research Council (ERC) under the European Union’s Horizon 2020 research and innovation programme (grant agreement No. 817762). The author was a predoctoral fellow of the Research Foundation -  Flanders (FWO), file number 1137921N. The current paper is based on part of the resulting PhD thesis \cite{mertens2022templicial}.
}

\makeatletter
\@namedef{subjclassname@2020}{
  \textup{2020} Mathematics Subject Classification}
\makeatother



\usepackage{graphicx}
\usepackage[english]{babel}
\usepackage{amsmath}
\usepackage{amssymb}
\usepackage{amsthm}
\usepackage[all,cmtip]{xy}
\usepackage{cancel}
\usepackage[alpine]{ifsym}
\usepackage{enumerate}
\usepackage{stmaryrd}
\usepackage{tikz}
\usepackage{tikz-cd}
\usetikzlibrary{matrix}
\usepackage[toc,page]{appendix}
\usepackage{hyperref}
\usepackage[style=alphabetic, maxnames = 50, maxalphanames=50, urldate=long]{biblatex}
\DeclareMathOperator{\Ob}{Ob}

\DeclareMathOperator{\id}{id}

\DeclareMathOperator{\Fun}{Fun}

\DeclareMathOperator{\Set}{Set}

\DeclareMathOperator{\Comon}{Comon}
\DeclareMathOperator{\Ab}{Ab}
\DeclareMathOperator{\Mod}{Mod}

\DeclareMathOperator{\Coalg}{Coalg}

\DeclareMathOperator{\SSet}{SSet}

\DeclareMathOperator{\Colax}{Colax}

\DeclareMathOperator{\Quiv}{Quiv}
\DeclareMathOperator{\Cat}{Cat}

\newcommand{\fint}{\mathbf{\Delta}_{f}} %category of finite intervals
\newcommand{\simp}{\mathbf{\Delta}} %simplicial walking category
\newcommand{\ts}{S_{\otimes}} %category of ts-objects
\newtheorem{Thm}{Theorem}[section]
\newtheorem*{Thm*}{Theorem}
\newtheorem{Lem}[Thm]{Lemma}
\newtheorem{Prop}[Thm]{Proposition}
\newtheorem{Cor}[Thm]{Corollary}

\theoremstyle{definition}
\newtheorem{Def}[Thm]{Definition}
\newtheorem{Ex}[Thm]{Example}
\newtheorem{Exs}[Thm]{Examples}
\newtheorem{Con}[Thm]{Construction}

\theoremstyle{remark}
\newtheorem{Rem}[Thm]{Remark}

\begin{document}

\begin{abstract}
We follow the work of Aguiar \cite{aguiar1997internal} on internal categories and introduce simplicial objects internal to a monoidal category as certain colax monoidal functors. Then we compare three approaches to equipping them with a discrete set of vertices. We introduce based colax monoidal functors and show that under suitable conditions they are equivalent to the templicial objects defined in  \cite{lowen2024enriched}. We also compare templicial objects to the enriched Segal precategories appearing in \cite{lurie2009goodwillie}, \cite{simpson2012homotopy} and \cite{bacard2010Segal}, and show that they are equivalent for cartesian monoidal categories, but not in general.
\end{abstract}

\maketitle

\setcounter{tocdepth}{2}

\tableofcontents

\section{Introduction}\label{section: Introduction}

In \cite{aguiar1997internal}, Aguiar introduced a theory of graphs and categories internal to a monoidal category $\mathcal{V}$. A \emph{graph internal to $\mathcal{V}$} consists of a comonoid $G_{0}$ in $\mathcal{V}$ (the object of objects) and a bicomodule $G_{1}$ over $G_{0}$ (the object of edges). Endowing the pair $(G_{0},G_{1})$ with an appropriate composition law and identities then produces the notion of a \emph{category internal to $\mathcal{V}$}. Let us denote the category of such internal categories by $\Cat_{\otimes}(\mathcal{V})$. Note that if $\mathcal{V}$ is cartesian monoidal - i.e. the monoidal product $\otimes$ is the cartesian product - then any object of $\mathcal{V}$ has a unique comonoid structure. More specifically, we have an equivalence of categories $\Comon(\mathcal{V})\simeq \mathcal{V}$. As a consequence, graphs and categories internal to $\mathcal{V}$ recover the usual notions of internal graphs and categories when $\mathcal{V}$ is cartesian. On the other hand, if the object of objects is assumed to be \emph{discrete} - that is $G_{0}\simeq \coprod_{a\in S}I$ for some set $S$ - then for suitable $\mathcal{V}$ this precisely recovers $\mathcal{V}$-enriched categories with object set $S$.

$$
\Cat(\mathcal{V})\underset{\text{if }\otimes = \times}{\simeq} \Cat_{\otimes}(\mathcal{V})\supseteq \mathcal{V}\Cat
$$

This picture extends nicely to higher dimensions as well, where we can consider the higher dimensional analogue of graphs to be simplicial sets $\SSet$. In order to do this, we must first reformulate the definition of a simplicial object a little, and introduce some coalgebraic structure just like for Aguiar's internal graphs.

Consider the category $\fint$ of \emph{finite intervals}. It is the subcategory of the usual simplex category $\simp$ containing all order morphisms $f: [m]\rightarrow [n]$ which preserve the endpoints, i.e. $f(0) = 0$ and $f(m) = n$. Unlike $\simp$, $\fint$ carries a monoidal structure $(+,[0])$ which is given by $[m] + [n] = [m+n]$ on objects. Then we have the following result by Leinster.

\begin{Prop}[\cite{leinster2000homotopy}, Proposition 3.1.7]\label{proposition: Leinster}
Let $\mathcal{V}$ be a cartesian monoidal category. Then we have an equivalence of categories
$$
\Colax(\fint^{op},\mathcal{V})\simeq S\mathcal{V}
$$
between colax monoidal functors $\fint^{op}\rightarrow \mathcal{V}$ and simplicial objects in $\mathcal{V}$.
\end{Prop}

For instance the classical nerve functor $N: \Cat\rightarrow \SSet$ straightforwardly generalizes to a functor $N_{\mathcal{V}}: \Cat_{\otimes}(\mathcal{V})\rightarrow \Colax(\fint^{op},\mathcal{V})$, whereas in general we do not obtain a nerve $\Cat_{\otimes}(\mathcal{V})\rightarrow S\mathcal{V}$ in the absence of cartesian projections. We will come back to this point in Section \ref{section: Why colax monoidal functors?}. This suggests that - even when $\mathcal{V}$ is not cartesian - we may interpret colax monoidal functors $\fint^{op}\rightarrow \mathcal{V}$ as \emph{simplicial objects internal to $\mathcal{V}$}.

The next question then presents itself naturally. In parallel to $\mathcal{V}$-enriched categories, how do we discretize the object of vertices $X_{0}$ of an internal simplicial object $X: \fint^{op}\rightarrow \mathcal{V}$? In this note, we present and compare three possible approaches.

\begin{enumerate}
\item The most obvious way to discretize the vertices of a simplicial object $X$ internal to $\mathcal{V}$ is to require $X_{0}\simeq \coprod_{a\in S}I$ for some set $S$. In the cartesian case, we obtain a category of \emph{based simplicial objects} $S_{b}\mathcal{V}$. In the general case we obtain a category $\Colax_{b}(\fint^{op},\mathcal{V})$ of \emph{based colax monoidal functors}, see Definition \ref{definition: based colax monoidal functor}.

\item Assume first that $\mathcal{V}$ is cartesian, so that we are again working with simplicial objects $S\mathcal{V}$. In \cite{lurie2009goodwillie}, Lurie introduced enriched Segal precategories in order to define an enriched variant of classical Segal categories \cite{hirschowitz1998descente}. They were further studied by Simpson in \cite{simpson2012homotopy}. In this case the discretization of the vertices is achieved by replacing the simplex category $\simp$ by a labelled analogue $\simp_{S}$ for a given set $S$. One then considers functors $\simp^{op}_{S}\rightarrow \mathcal{V}$ instead, yielding the category $\mathbf{PC}(\mathcal{V})$, see Definition \ref{definition: enriched Segal precat.}. Building on this, Bacard considered a generalization of this approach for general (not necessarily cartesian) $\mathcal{V}$ in \cite{bacard2010Segal}, also using colax functors. These approaches are discussed in Section \ref{section: Templicial objects versus enriched Segal precategories}.

\item In \cite{lowen2024enriched}, templicial objects were introduced by Wendy Lowen and the present author, and this note represents part of a bigger project to study them. They are simplicial objects internal to $\mathcal{V}$ where the discrete set of vertices is realized by means of $\mathcal{V}$-enriched quivers. So we consider colax monoidal functors of the form $\fint^{op}\rightarrow \mathcal{V}\Quiv_{S}$ and denote their category by $\ts\mathcal{V}$, see Definition \ref{definition: templicial obj.}. 
\end{enumerate}

Throughout the text, we let $(\mathcal{V},\otimes,I)$ be a fixed bicomplete, symmetric monoidal closed category (i.e. a B\'enabou cosmos in the sense of \cite{street1974elementary}). 

In Section \ref{section: Why colax monoidal functors?}, we make the case that passing to colax monoidal functors $\fint^{op}\rightarrow \mathcal{V}$ is necessary in the context of non-cartesian monoidal categories. The main argument is that this allows to define well-behaved nerves of $\mathcal{V}$-enriched categories, which simplicial objects do not.

In Section \ref{section: Templicial objects versus based colax monoidal functors}, we compare approaches $(1)$ and $(3)$, and we identify conditions on $\mathcal{V}$ for which they are equivalent. We call $\mathcal{V}$ \emph{decomposing} if it satisfies these conditions (Definition \ref{definition: decomposing monoidal category}) and the equivalence $\ts\mathcal{V}\simeq \Colax_{b}(\fint^{op},\mathcal{V})$ is proven in Theorem \ref{theorem: comparison is equiv.}.

In Section \ref{section: Templicial objects versus enriched Segal precategories}, we compare approaches $(2)$ and $(3)$. In the cartesian case, in \cite{simpson2012homotopy}, Simpson introduced the conditions (DISJ) on $\mathcal{V}$ under which $\mathbf{PC}(\mathcal{V})\simeq S_{b}\mathcal{V}$ (see Defintiion \ref{definition: DISJ}). We show that (DISJ) implies decomposing for $\mathcal{V}$ (Proposition \ref{proposition: DISJ implies decomposing}), whence under (DISJ) all three approaches become equivalent (Corollary \ref{corollary: temp. objs. vs precats.}). Finally, in the non-cartesian case we cannot expect such an equivalence to hold even for decomposing $\mathcal{V}$, and we show in \S\ref{subsection: The non-cartesian case} that for $\mathcal{V} = \Mod(k)$, the templicial dg-nerve $dg\Cat_{k}\rightarrow \ts\Mod(k)$ from \cite{lowen2023frobenius} does not factor through $\mathbf{PC}(\Mod(k))$ (see Example \ref{example: failure of dg-nerve for precats}).\\

\noindent \emph{Acknowledgement.} The author is grateful to Boris Shoikhet for pointing out [Bac10] and to Lander Hermans for [Agu97], as well as interesting discussions on the subject. The author would also like to thank Violeta Borges Marques for discussions on the counterexamples at the end of this note and Wendy Lowen for valuable feedback. Further thanks are extended to Bernhard Keller, Tom Leinster, Michel Van den Bergh and Ittay Weiss for interesting comments and questions on the project.

\section{Why colax monoidal functors?}\label{section: Why colax monoidal functors?}

Recall the classical nerve functor from small categories to simplicial sets.
$$
N: \Cat\hookrightarrow \SSet
$$
Given a small category $\mathcal{C}$, $N(\mathcal{C})$ is the simplicial set whose set of $n$-simplices is given by
\begin{equation}\label{equation: classical nerve}
N(\mathcal{C})_{n} = \coprod_{A_{0},...,A_{n}\in \Ob(\mathcal{C})}\mathcal{C}(A_{0},A_{1})\times ...\times \mathcal{C}(A_{n-1},A_{n})
\end{equation}
The inner face maps $d_{j}$ and degeneracy maps $s_{i}$ are defined by
\begin{align*}
d_{j}&: X_{n}\rightarrow X_{n-1}: (f_{1},\dots,f_{n})\mapsto (f_{1},\dots,f_{i}\circ f_{i-1},\dots,f_{n})\\
s_{i}&: X_{n}\rightarrow X_{n+1}: (f_{1},\dots,f_{n})\mapsto (f_{1},\dots,f_{i-1},\id,f_{i},\dots,f_{n})
\end{align*}
for all $0 < j < n$ and $0\leq i\leq n$. Thus the $d_{j}$ compose two consecutive morphisms in a sequence and the $s_{i}$ insert an identity. The outer face maps $d_{0}$ and $d_{n}$ are defined by deleting either the first or the last entry in a sequence, e.g.
$$
d_{0}: X_{n}\rightarrow X_{n-1}: (f_{1},\dots,f_{n})\mapsto (f_{2},\dots,f_{n})
$$

Now for a $\mathcal{V}$-enriched category, we can wonder what the nerve of $\mathcal{C}$ should be. At first, one might guess that this should yield a simplicial object $N(\mathcal{C})\in S\mathcal{V}$. Analogous to \eqref{equation: classical nerve}, we put (for $n = 0$ this yields $N(\mathcal{C})_{0} = \coprod_{A\in \Ob(\mathcal{C})}I$):
\begin{equation}\label{equation: enriched nerve}
N(\mathcal{C})_{n} = \coprod_{A_{0},...,A_{n}\in \Ob(\mathcal{C})}\mathcal{C}(A_{0},A_{1})\otimes ...\otimes \mathcal{C}(A_{n-1},A_{n})\in \mathcal{V}
\end{equation}
It is easy to see that we can define inner face morphisms and degeneracy morphisms in the same way as above. However, for a general (non-cartesian) monoidal category $\mathcal{V}$, we cannot define the outer face maps because there are no projection maps
\begin{equation}\label{diagram: failure of outer face map}
\mathcal{C}(A_{0},A_{1})\otimes\mathcal{C}(A_{1},A_{2})\otimes \dots\otimes \mathcal{C}(A_{n-1},A_{n})\rightarrow \mathcal{C}(A_{1},A_{2})\otimes\dots\otimes \mathcal{C}(A_{n-1},A_{n})
\end{equation}
Hence, we do not obtain a simplicial object in this way. In fact, assuming some compatibility with the classical nerve, no other definition of $N(\mathcal{C})$ will do the trick either.

\begin{Rem}
Let $k$ be a unital commutative ring and set $\mathcal{V} = \Mod(k)$ with the tensor product $\otimes$ of $k$-modules. Then $\mathcal{V}$ is certainly not cartesian monoidal. Consider the category $k\Cat$ of small $k$-linear (that is, $\Mod(k)$-enriched) categories. Then there is \emph{no} functor
$$
N_{k}: k\Cat\rightarrow S\Mod(k)
$$
such that the following commutative square commutes (up to natural isomorphism), where the vertical arrows represent the forgetful functors.
$$
\xymatrix{
k\Cat\ar[r]^(0.4){N_{k}}\ar[d]_{\mathcal{U}} & {S\Mod(k)}\ar[d]^{U} \\
\Cat\ar[r]_{N} & \SSet
}
$$
In fact, it doesn't even commute after applying the homotopy category functor $h: \SSet\rightarrow \Cat$ (which is left-adjoint to $N$). Indeed, otherwise we would have an equivalence of categories
$$
hUN_{k}(\mathcal{C})\simeq hNU(\mathcal{C})\simeq U(\mathcal{C})
$$
for all small $k$-linear categories $\mathcal{C}$. So since $UN_{k}(\mathcal{C})\in S\Mod(k)$ is always a Kan complex, $U(\mathcal{C})$ is a groupoid. But no $k$-linear category is a groupoid unless it is trivial.
\end{Rem}

One might object that if $\mathcal{V}$ is \emph{semi-cartesian} (i.e. the monoidal unit $I$ is a terminal object), then one can still define outer face maps \eqref{diagram: failure of outer face map} by means of the unique morphisms $\mathcal{C}(A,B)\rightarrow I$. Indeed, this will still yield a well-defined simplicial object $N'(\mathcal{C})\in S\mathcal{V}$ and we thus obtain a functor
$$
N': \mathcal{V}\Cat\rightarrow S\mathcal{V}
$$
as long as $\mathcal{V}$ is semi-cartesian. However, this nerve functor will in general not be full.

\begin{Ex}\label{example: naive nerve is not full}
Consider the $\mathbb{Z}$-algebra $A = \mathbb{Z}[X]$ with augmentation $\epsilon: A\rightarrow \mathbb{Z}: X\mapsto 0$.
Then $A$ can be considered as a one-object category enriched over $\mathcal{V} = \Mod(\mathbb{Z})/\mathbb{Z}$, which is semi-cartesian. Then we have, for all $n\geq 0$:
$$
N(A)_{n} = A^{\otimes n}\simeq \mathbb{Z}[X_{1},...,X_{n}]
$$
with augmentation $\mathbb{Z}[X_{1},...,X_{n}]\rightarrow \mathbb{Z}: X_{i}\mapsto 0$. The face and degeneracy maps for $0\leq j\leq n$ are in fact ring morphisms and on the $X_{i}$ ($1\leq i\leq n$) they are given by:
\begin{align*}
d_{j}(X_{i}) = 
\begin{cases}
X_{i} & \text{if }i < j\text{ or }i = j < n\\
X_{i-1} & \text{if }i > j\\
0 & \text{otherwise}
\end{cases}
\quad \text{and}\quad 
s_{j}(X_{i})\mapsto
\begin{cases}
X_{i} & \text{if }i < j\\
X_{i+1} & \text{if }i\geq j\\
\end{cases}
\end{align*}
We define a simplicial $\mathbb{Z}$-linear map $\alpha: N(A)\rightarrow N(A)$ by setting
$$
\alpha_{n}(X^{i_{1}}_{1}\cdots X^{i_{n}}_{n}) =
\begin{cases}
X^{i_{1}}_{1}\cdots X^{i_{n}}_{n} & \text{if }i_{1} + ... + i_{n}\leq 1\\
0 & \text{otherwise}
\end{cases}
$$
for all $n\geq 0$ and $i_{1},...,i_{n}\geq 0$. It is easily verified to be compatible with the face and degeneracy maps. Moreover, it clearly respects the augmentations.

Now for any ring homomorphism $f: A\rightarrow A$, $N(f) = \alpha$ would imply that $f = \alpha_{1}$. But $\alpha_{1}$ is not a ring homomorphism. Hence, the functor $N': \mathcal{V}\Cat\rightarrow S\mathcal{V}$ is not full.
\end{Ex}

However, the above data \eqref{equation: enriched nerve} can be neatly organised into a colax monoidal functor
$$
X: \fint^{op}\rightarrow \mathcal{V}
$$
which still has inner face maps and degeneracy maps, but the outer face maps are replaced by comultiplication maps. This construction was investigated in detail in the context of templicial objects in \cite{lowen2024enriched}. The resulting nerve functor has many desirable properties.

\begin{Prop}[\cite{lowen2024enriched}, \S 2.B]
There exists a fully faithful right-adjoint functor
$$
N_{\mathcal{V}}: \mathcal{V}\Cat\hookrightarrow \ts\mathcal{V}
$$
such that $\tilde{U}\circ N_{\mathcal{V}}\simeq N\circ \mathcal{U}$, where $\tilde{U}: \ts\mathcal{V}\rightarrow \SSet$ is the underlying simplicial set functor. Moreover, $N_{\mathcal{V}}$ recovers the classical nerve $N: \Cat\rightarrow \SSet$ when $\mathcal{V} = \Set$.
\end{Prop}

\section{Templicial objects versus based colax monoidal functors}\label{section: Templicial objects versus based colax monoidal functors}

In this section, we introduce based colax monoidal functors and compare them to the templicial objects of \cite{lowen2024enriched}. We define a based colax monoidal functor in \S\ref{subsection: Based colax monoidal functors} as an internal simplicial object $X: \fint^{op}\rightarrow \mathcal{V}$ such that the comonoid $X_{0}$ is free on a set, its \emph{base}. Then we recall templicial objects in \S\ref{subsection: Templicial objects}. In \S\ref{subsection: Decomposing monoidal categories} we identify conditions on the monoidal category $\mathcal{V}$ for which these notions become equivalent, and call $\mathcal{V}$ \emph{decomposing} if it satisfies them. The equivalence $\Colax_{b}(\fint^{op},\mathcal{V})\simeq \ts\mathcal{V}$ is then shown in \S\ref{subsection: Comparison}.

\subsection{Based colax monoidal functors}\label{subsection: Based colax monoidal functors}

Let us define $S_{b}\mathcal{V}$ as the $2$-pullback
\begin{equation}\label{diagram: based simp. objs.}
\begin{tikzcd}
	{S_{b}\mathcal{V}} & {S\mathcal{V}} \\
	\Set & {\mathcal{V}}
	\arrow[from=1-1, to=2-1]
	\arrow[from=1-1, to=1-2]
	\arrow["{(-)_{0}}", from=1-2, to=2-2]
	\arrow["F"', from=2-1, to=2-2]
\end{tikzcd}
\end{equation}
In other words, $S_{b}\mathcal{V}$ is the category of simplicial objects $X$ in $\mathcal{V}$ along with a set $S$ and an isomorphism $X_{0}\simeq F(S)$ in $\mathcal{V}$. We can make a similar construction for simplicial objects internal to $\mathcal{V}$ as a monoidal category.

\begin{Rem}\label{remark: comonoids in a cartesian category}
Note that if $\mathcal{V}$ is cartesian monoidal, every object $A\in \mathcal{V}$ has a unique comonoid structure with the diagonal $\Delta: A\rightarrow A\times A$ as comultiplication and the terminal map $t: A\rightarrow 1$ as counit. This extends to an equivalence of categories:
$$
\mathcal{V}\simeq \Comon(\mathcal{V})
$$
In particular, this applies for the cartesian category of sets $(\Set,\times,\{*\})$.

Now for a general monoidal category $(\mathcal{V},\otimes,I)$, the free functor $F: \Set\rightarrow \mathcal{V}: S\mapsto \coprod_{a\in S}I$ is strong monoidal and thus we have an induced functor
$$
F: \Set\simeq \Comon(\Set)\rightarrow \Comon(\mathcal{V})
$$
\end{Rem}

\begin{Def}\label{definition: based colax monoidal functor}
Let $(X,\mu,\epsilon): \fint^{op}\rightarrow \mathcal{V}$ be a colax monoidal functor. Then $X_{0}$ has the structure of a comonoid with comultiplication given by $\mu_{0,0}: X_{0}\rightarrow X_{0}\otimes X_{0}$ and counit given by $\epsilon: X_{0}\rightarrow I$. We call a set $S$ a \emph{base} of $X$ if it comes equipped with an isomorphism of comonoids $\varphi: X_{0}\xrightarrow{\sim} F(S)$. We call the triple $(X,S,\varphi)$ a \emph{based colax monoidal functor}.

Consider the functor
$$
(-)_{0}: \Colax(\fint^{op},\mathcal{V})\rightarrow \Comon(\mathcal{V}): (X,\mu,\epsilon)\mapsto (X_{0},\mu_{0,0},\epsilon)
$$
We define the category $\Colax_{b}(\fint^{op},\mathcal{V})$ by the $2$-pullback
\[\begin{tikzcd}
	{\Colax_{b}(\fint^{op},\mathcal{V})} & {\Colax(\fint^{op},\mathcal{V})} \\
	\Set & {\Comon(\mathcal{V})}
	\arrow[from=1-1, to=2-1]
	\arrow[from=1-1, to=1-2]
	\arrow["{(-)_{0}}", from=1-2, to=2-2]
	\arrow["F"', from=2-1, to=2-2]
\end{tikzcd}\]
Note that its objects are precisely the based colax monoidal functors.

A morphism $X\rightarrow Y$ in $\Colax_{b}(\fint^{op},\mathcal{V})$ with respective bases $S$ and $T$ is a monoidal natural transformation $\alpha$ together with a map of sets $f: S\rightarrow T$ such that through the isomorphisms $X_{0}\simeq F(S)$ and $Y_{0}\simeq F(T)$, $\alpha_{0}$ is induced by $f$.
\end{Def}

\begin{Rem}\label{remark: based colax functors in a cartesian category}
Note that by Proposition \ref{proposition: Leinster} and Remark \ref{remark: comonoids in a cartesian category}, we have an equivalence of categories when $\mathcal{V}$ is cartesian monoidal:
$$
S_{b}\mathcal{V}\simeq \Colax_{b}(\fint^{op},\mathcal{V})
$$
In particular, if $\mathcal{V}$ is also decomposing (Definition \ref{definition: decomposing monoidal category}), we will have an equivalence of categories $\ts\mathcal{V}\simeq S_{b}\mathcal{V}$ by Theorem \ref{theorem: comparison is equiv.} below.
\end{Rem}

\subsection{Templicial objects}\label{subsection: Templicial objects}

The category $\ts\mathcal{V}$ of templicial objects in $\mathcal{V}$ was introduced in \cite{lowen2024enriched} as a generalization of simplicial sets, in order to define an enriched variant of Joyal's quasi-categories \cite{joyal2002quasi}. Before we can define them formally however, we must discuss some preliminaries about categories of enriched quivers.

Given a set $S$, we denote by $\mathcal{V}\Quiv_{S}$ the category of \emph{$\mathcal{V}$-enriched quivers with vertex set $S$}. That is, its objects are collections $Q = (Q(a,b))_{a,b\in S}$ of objects $Q(a,b)\in \mathcal{V}$ and a morphism $f: Q\rightarrow P$ is a collection of morphisms $(f_{a,b}: Q(a,b)\rightarrow P(a,b))_{a,b\in S}$ in $\mathcal{V}$. The category $\mathcal{V}\Quiv_{S}$ carries a monoidal structure $(\otimes_{S},I_{S})$ given as follows:
$$
(Q\otimes_{S} P)(a,b) = \coprod_{c\in S}Q(a,c)\otimes P(c,b)\quad \text{and}\quad I_{S}(a,b) =
\begin{cases}
I & \text{if }a = b\\
0 & \text{if }a\neq b
\end{cases}
$$
for all $Q,P\in \mathcal{V}\Quiv_{S}$ and $a,b\in S$. Further, given a map of sets $f: S\rightarrow T$. We have an induced lax monoidal functor $f^{*}: \mathcal{W}\Quiv_{T}\rightarrow \mathcal{W}\Quiv_{S}$ given by $f^{*}(Q)(a,b) = Q(f(a),f(b))$ for all $\mathcal{W}$-enriched quivers $Q$ and $a,b\in S$. The functor $f^{*}$ has a left-adjoint which we denote by $f_{!}: \mathcal{W}\Quiv_{S}\rightarrow \mathcal{W}\Quiv_{T}$. As $f^{*}$ is lax monoidal, $f_{!}$ comes equipped with an induced colax monoidal structure.

Finally, we turn to templicial objects. As was explained in Remark 2.6 of \cite{lowen2024enriched}, their category $\ts\mathcal{V}$ can be described by a Grothendieck construction as follows. Let $\underline{\Cat}$ denote the (very large) $2$-category of (large) categories, functors and natural transformations. We have a pseudofunctor
\begin{equation}\label{pseudofunctor}
\Phi = \Colax(\fint^{op}, (-)_{!}): \Set\rightarrow \underline{\Cat}: S\mapsto \Colax(\fint^{op},\mathcal{V}\Quiv_{S})
\end{equation}
which sends a map of sets $f: S\rightarrow T$ to the functor $f_{!}\circ -: \Colax(\fint^{op},\mathcal{V}\Quiv_{S})\rightarrow \Colax(\fint^{op},\mathcal{V}\Quiv_{T})$ which post-composes with the colax monoidal functor $f_{!}$ from above. The Grothendieck construction $\int\Phi$ is then the category of pairs $(X,S)$ with $S$ a set and $X: \fint^{op}\rightarrow \mathcal{V}\Quiv_{S}$ a colax monoidal functor.

\begin{Def}\label{definition: templicial obj.}
We call an object $(X,S)\in \int\Phi$ a \emph{tensor-simplicial} or \emph{templicial object} in $\mathcal{V}$ if $X$ is \emph{strongly unital}, i.e. if the counit $\epsilon: X_{0}\rightarrow I_{S}$ is an isomorphism. We denote the full subcategory of $\int\Phi$ spanned by all templicial objects by $\ts\mathcal{V}$.
\end{Def}

For more details on templicial objects, we refer to \cite{lowen2023frobenius}\cite{lowen2024enriched}.

\subsection{Decomposing monoidal categories}\label{subsection: Decomposing monoidal categories}

\begin{Def}\label{definition: decomposing equalizer}
Let $\mathcal{C}$ be a category with coproducts. Let $S$ be a set and $A\in \mathcal{C}$. We denote $\iota_{j}: A\rightarrow \coprod_{j\in S}A$ for the $j$th coprojection and $\nabla: \coprod_{i\in S}A\rightarrow A$ for the codiagonal. We call a morphism $f: A\rightarrow \coprod_{i\in S}A$ \emph{decomposing} if
\begin{equation}\label{equation: decomposing morphism}
\left(\coprod_{i\in S}f\right)f = \left(\coprod_{i\in S}\iota_{i}\right)f\quad \text{and}\quad \nabla f = \id_{A}
\end{equation}
A \emph{decomposing equalizer} is the equalizer of a decomposing morphism with a coprojection $\iota_{j}$ for some $j\in S$.
\end{Def}

\begin{Ex}\label{example: coprojection is decomposing}
Any coprojection $\iota_{j}: A\rightarrow \coprod_{i}A$ is itself decomposing.
\end{Ex}

\begin{Rem}\label{remark: decomp. eq. are coreflexive}
Note that because of the condition $\nabla f = \id_{A}$, a decomposing equalizer is always coreflexive.
\end{Rem}

\begin{Lem}\label{lemma: combined decomposing equalizer is split}
Let $\mathcal{C}$ be a category with coproducts and consider a decomposing morphism $f: A\rightarrow \coprod_{i\in S}A$. Then
\[\begin{tikzcd}
	A & {\underset{i\in S}{\coprod}A} & {\underset{i\in S}{\coprod}\underset{j\in S}{\coprod}A}
	\arrow["{\coprod_{i}f}", shift left=1, from=1-2, to=1-3]
	\arrow["{\coprod_{i}\iota_{i}}"', shift right=1, from=1-2, to=1-3]
	\arrow["f", from=1-1, to=1-2]
\end{tikzcd}\]
is a split equalizer.
\end{Lem}
\begin{proof}
Let $\overline{\nabla}: \coprod_{i}\coprod_{j}A\rightarrow \coprod_{j}A$ denote the codiagonal which collapses the outer coproduct. Then it immediately follows that $\overline{\nabla}\coprod_{i}f = f\nabla$ and $\overline{\nabla}\coprod_{i}\iota_{i} = \id$. By hypothesis, we also have $\nabla f = \id_{A}$.
\end{proof}

Recall that a coproduct $\coprod_{i\in I}A_{i}$ of objects in $\mathcal{V}$ is called \emph{disjoint} if all coprojections $\iota_{j}: A_{j}\rightarrow \coprod_{i\in I}A_{i}$ are monomorphisms and the intersection of $A_{i}$ and $A_{j}$ is initial whenever $i\neq j$.

\begin{Def}\label{definition: decomposing monoidal category}
We call $\mathcal{V}$ \emph{decomposing} if it satisfies the following conditions:
\begin{enumerate}[(a)]
\item coproducts commute with decomposing equalizers in $\mathcal{V}$,
\item coproducts are disjoint in $\mathcal{V}$, 
\item the monoidal product $-\otimes -$ of $\mathcal{V}$ preserves decomposing equalizers in each variable.
\end{enumerate}
\end{Def}

\begin{Ex}
In a cartesian category $\mathcal{V}$, the product $-\times -$ commutes with all equalizers. So if we assume that coproducts are disjoint and commute with equalizers, then $\mathcal{V}$ is decomposing.

This is the case for any Grothendieck topos and the categories $\mathrm{Top}$ of topological spaces, $\Cat$ of small categories and $\mathrm{Poset}$ of posets for example.
\end{Ex}

\begin{Lem}\label{lemma: decomposing equalizer in additive category is split}
If $\mathcal{C}$ is an additive category, any decomposing equalizer in $\mathcal{C}$ is split.
\end{Lem}
\begin{proof}
Let $f: A\rightarrow \bigoplus_{i\in S}A$ be a decomposing morphism in $\mathcal{C}$ and fix $j\in S$. Consider the equalizer $e: E\rightarrow A$ of $f$ and $\iota_{j}$. Then for the $j$th projection $p: \bigoplus_{i\in S}A\rightarrow A$ we have $p\iota_{j} = \id_{A}$ and
$$fpf = p'\left(\bigoplus_{i\in S}f\right)f = p'\left(\bigoplus_{i\in S}\iota_{i}\right)f = \iota_{j}pf$$
where $p': \bigoplus_{i,k}A\rightarrow \bigoplus_{k}A$ is the projection onto the component $i = j$. So there exists a unique $s: A\rightarrow E$ such that $es = pf$. Then, $ese = pfe = p\iota_{j}e = e$ and thus $se = \id_{E}$ because $e$ is a monomorphism.
\end{proof}

\begin{Prop}\label{proposition: preadditive implies decomposing}
If $\mathcal{V}$ is additive, then $\mathcal{V}$ is decomposing.
\end{Prop}
\begin{proof}
By Lemma \ref{lemma: decomposing equalizer in additive category is split}, decomposing equalizers in $\mathcal{V}$ are split equalizers and are thus preserved by all functors. In particular, both the coproduct functor $\mathcal{V}^{S}\rightarrow \mathcal{V}$ and the monoidal product $-\otimes -$ preserve decomposing equalizers. Further, in an $\Ab$-enriched category, coproducts are always disjoint.
\end{proof}

\begin{Prop}\label{proposition: stability of decomposing under constructions}
Assume $\mathcal{V}$ is decomposing.
\begin{enumerate}[1.]
\item For any monoid $M$ in $\mathcal{V}$, the overcategory $\mathcal{V}/M$ with the induced monoidal structure is decomposing.
\item The category $\Comon(\mathcal{V})$ of comonoids in $\mathcal{V}$ with its induced monoidal structure is decomposing.
\end{enumerate}
\end{Prop}
\begin{proof}
\begin{enumerate}[1.]
\item This immediately follows from the fact that the forgetful functor $\mathcal{V}/M\rightarrow \mathcal{V}$ is strong monoidal and creates colimits and equalizers.
\item The forgetful functor $\Comon(\mathcal{V})\rightarrow \mathcal{V}$ is strong monoidal and creates colimits. It will suffice that it also creates decomposing equalizers. But the latter follows from the fact that $-\otimes -$ preserves decomposing equalizers in $\mathcal{V}$.
\end{enumerate}
\end{proof}

\subsection{Comparison}\label{subsection: Comparison}

We now describe a comparison functor from templicial objects to based colax monoidal functors and show that it is an equivalence when $\mathcal{V}$ is a decomposing monoidal category.

\begin{Con}\label{construction: comparison functor}
Consider the natural transformation $t: \id_{\Set}\rightarrow *$ given by the terminal map $t_{S}: S\rightarrow \{*\}$ for every set $S$. This induces a pseudonatural transformation
$$
\Phi t: \Phi\rightarrow \Phi\circ *
$$
between pseudofunctors $\Set\rightarrow \Cat$, where $\Phi = \Colax(\fint^{op},(-)_{!})$ is as in \eqref{pseudofunctor}. Through the Grothendieck construction, we obtain a functor
$$
\mathfrak{c} : \int\Phi\rightarrow \int\Phi\circ *\simeq \Colax(\fint^{op},\mathcal{V})\times \Set
$$

Explicitly, this functor sends a pair $(X,S)$ with $S$ a set and $X: \fint^{op}\rightarrow \mathcal{V}\Quiv_{S}$ colax monoidal to the pair $(\mathfrak{c}X,S)$, where
$$
\mathfrak{c}X_{n} = (t_{S})_{!}(X_{n}) = \coprod_{a,b\in S}X_{n}(a,b)
$$
for all $n\geq 0$. The comultiplication and counit are induced by those of $X$. Moreover, a templicial morphism $(\alpha,f): (X,S)\rightarrow (Y,T)$ is sent to the pair $(\mathfrak{c}\alpha, f)$, where for every $n\geq 0$,
$$
\mathfrak{c}\alpha_{n}: \coprod_{a,b\in S}X_{n}(a,b)\rightarrow \coprod_{x,y\in T}Y_{n}(x,y)
$$
is induced by $(\alpha_{n})_{a,b}: X_{n}(a,b)\rightarrow Y_{n}(f(a),f(b))$ for all $a,b\in S$.
\end{Con}

Note that, up to equivalence, we may consider $\Colax_{b}(\fint^{op},\mathcal{V})$ as a subcategory of $\Colax(\fint^{op},\mathcal{V})\times \Set$.

\begin{Prop}
The functor $\mathfrak{c}: \int\Phi\rightarrow \Colax(\fint^{op},\mathcal{V})\times \Set$ of Construction \ref{construction: comparison functor} restricts to a functor
$$
\mathfrak{c}: \ts\mathcal{V}\rightarrow \Colax_{b}(\fint^{op},\mathcal{V})
$$
\end{Prop}
\begin{proof}
Note that for any set $S$, $(t_{S})_{!}(I_{S})\simeq \coprod_{x\in S}I = F(S)$. Take an object $(X,S)$ of $\int\Phi$, then the counit $\epsilon: X_{0}\rightarrow I_{S}$ induces a morphism
\begin{displaymath}
\varphi_{(X,S)}: \mathfrak{c}X_{0} = (t_{S})_{!}(X_{0})\rightarrow F(S)
\end{displaymath}
in $\mathcal{V}$. It easily follows that $\varphi_{(X,S)}$ is a comonoid morphism which is natural in $(X,S)$. Moreover, if $(X,S)$ is a templicial object, then $\epsilon$ and thus $\varphi_{(X,S)}$ is an isomorphism.
\end{proof}

We now describe how to invert the comparison functor $\mathfrak{c}: \ts\mathcal{V}\rightarrow \Colax_{b}(\fint^{op},\mathcal{V})$. For this we need to ``decompose'' the objects $X_{n}\in\mathcal{V}$ of a based colax monoidal functor to form a quiver. This goes as follows.

\begin{Con}\label{construction: decomposing equalizer of cc-functors}
Let $X: \fint^{op}\rightarrow \mathcal{V}$ be a based colax monoidal functor with comultiplication $\mu$ and base $S$. Via the isomorphism $X_{0}\simeq F(S)\simeq \coprod_{a\in S}I$, we have for every $n\geq 0$, a morphism
$$
\mu_{0,n,0}: X_{n}\rightarrow X_{0}\otimes X_{n}\otimes X_{0}\simeq \coprod_{a,b\in S}X_{n}
$$
which assemble into a natural transformation $\mu_{0,-,0}: X\rightarrow \coprod_{a,b\in S}X$.

Then define $X(a,b)$ as the equalizer
\[\begin{tikzcd}
	{X(a,b)} & X & {\underset{a,b\in S}{\coprod}X}
	\arrow["{\mu_{0,-,0}}", shift left=1, from=1-2, to=1-3]
	\arrow["{c_{a,b}}"', shift right=1, from=1-2, to=1-3]
	\arrow["{e_{a,b}}", from=1-1, to=1-2]
\end{tikzcd}\]
in $\Fun(\fint^{op},\mathcal{V})$, where $c_{a,b}$ is the $(a,b)$th coprojection. This is in fact a decomposing equalizer in the sense of Definition \ref{definition: decomposing equalizer}, as shown by the following lemma.
\end{Con}

\begin{Lem}\label{lemma: decomposing equalizers of cc-functors}
Let $X$ be a based colax monoidal functor with comultiplication $\mu$ and base $S$. Then $\mu_{0,-,0}: X\rightarrow \coprod_{a,b\in S}X$ is a decomposing morphism of $\Fun(\fint^{op},\mathcal{V})$.
\end{Lem}
\begin{proof}
Note that through the isomorphism $X_{0}\simeq \coprod_{a\in S}I$, the counit $\epsilon: X_{0}\rightarrow I$ becomes the codiagonal. Moreover, for all $n\geq 0$, the morphisms $\id_{X_{0}}\otimes \mu_{0,n,0}\otimes \id_{X_{0}}$ and $\mu_{0,0}\otimes \id_{X_{n}}\otimes \mu_{0,0}$ become $\coprod_{a,b}\mu_{0,n,0}$ and $\coprod_{a,b}c_{a,b}$ respectively. Thus the conditions \eqref{equation: decomposing morphism} for $\mu_{0,-,0}$ to be a decomposing morphism precisely translate to the coassociativity of $\mu$ and its counitality with $\epsilon$.
\end{proof}

\begin{Prop}\label{proposition: decomposition of Fts-object}
Suppose that $\mathcal{V}$ is a decomposing monoidal category. Let $X$ be a based colax monoidal functor with base $S$, comultiplication $\mu$ and counit $\epsilon$. Then:
\begin{enumerate}[1.]
\item The canonical natural transformation
$$
(e_{a,b})_{a,b}:\coprod_{a,b\in S}X(a,b)\rightarrow X
$$
is an isomorphism.

\item For all $a,b\in S$, the composition
$$
X_{0}(a,a)\xrightarrow{e_{a,a}} X_{0}\xrightarrow{\epsilon} I
$$
is an isomorphism, and $X_{0}(a,b)\simeq 0$ if $a\neq b$.
\item For all $k,l\geq 0$ and $a,b\in S$, the composite $\mu_{k,l}e_{a,b}$ factors uniquely as
$$
X_{k+l}(a,b)\xrightarrow{\mu^{a,b}_{k,l}} \coprod_{c\in S}X_{k}(a,c)\otimes X_{l}(c,b)\xrightarrow{(e_{a,c}\otimes e_{c,b})_{c}} X_{k}\otimes X_{l}
$$
\end{enumerate}
\end{Prop}
\begin{proof}
\begin{enumerate}[1.]
\item By Lemma \ref{lemma: decomposing equalizers of cc-functors}, $\mu_{0,-,0}$ is decomposing and thus $\coprod_{a,b\in S}X(a,b)$ is the equalizer of $\coprod_{a,b}\mu_{0,-,0}$ and $\coprod_{a,b}c_{a,b}$. Hence by Lemma \ref{lemma: combined decomposing equalizer is split}, it is isomorphic to $X$. More precisely, for the isomorphism $\varphi: \coprod_{a,b}X(a,b)\xrightarrow{\sim} X$ we have $\coprod_{a,b}e_{a,b} = \mu_{0,-,0}\varphi$ and thus as $\epsilon$ coincides with the codiagonal $\nabla: \coprod_{a}I\rightarrow I$, we get $\varphi = (e_{a,b})_{a,b}$.

\item As coproducts are disjoint we have an equalizer diagram
\[\begin{tikzcd}
	{I_{a,x,b}} & I & {\underset{y,z\in S}{\coprod}I}
	\arrow["{\iota_{x,x}}", shift left=1, from=1-2, to=1-3]
	\arrow["{\iota_{a,b}}"', shift right=1, from=1-2, to=1-3]
	\arrow["f", from=1-1, to=1-2]
\end{tikzcd}\]
where $I_{a,x,b} = I$ if $a = b = x$ and $I_{a,x,b} = 0$ otherwise. Taking the coproduct of this diagram over all $x\in S$, we find an equalizer
\[\begin{tikzcd}
	{I_{a,b}} & {\underset{x\in S}{\coprod}I} & {\underset{y,x,z\in S}{\coprod}I}
	\arrow["{\coprod_{x}\iota_{x,x}}", shift left=1, from=1-2, to=1-3]
	\arrow["{\coprod_{x}\iota_{a,b}}"', shift right=1, from=1-2, to=1-3]
	\arrow["f", from=1-1, to=1-2]
\end{tikzcd}\]
where $I_{a,b} = I$ if $a = b$ and $I_{a,b} = 0$ if $a\neq b$. Now via the isomorphism $X_{0}\simeq \coprod_{x}I$, $\mu_{0,0,0}$ becomes $\coprod_{x}\iota_{x,x}$ and thus we have an isomorphism $\varphi: X_{0}(a,b)\rightarrow I_{a,b}$ such that $\iota_{a,b}\varphi = e_{a,b}$. As $\epsilon$ coincides with the codiagonal $\nabla$, we find that $\varphi = \epsilon e_{a,b}$.

\item Note that since decomposing equalizers are coreflexive (Remark \ref{remark: decomp. eq. are coreflexive}), and they are preserved by $-\otimes -$ in each variable, they are also preserved in both variables simultaneously. It then follows from Lemma \ref{lemma: decomposing equalizers of cc-functors} that the morphism
$$
\coprod_{c\in S}X_{k}(a,c)\otimes X_{l}(c,b)\xrightarrow{\coprod_{c}e_{a,c}\otimes e_{c,b}}\coprod_{c\in S}X_{k}\otimes X_{l}
$$
is the equalizer of $\coprod_{c}\mu_{0,k,0}\otimes \mu_{0,l,0}$ and $\coprod_{c}c_{a,c}\otimes c_{c,b}$. Using the isomorphism $X_{0}\simeq\coprod_{c}I$, we see that this is equivalently the equalizer of $\mu_{0,k,0}\otimes \id_{X_{0}}\otimes \mu_{0,l,0}$ and $c_{a,*}\otimes \mu_{0,0,0}\otimes c_{*,b}$, where
$$
c_{a,*}: X_{k}\simeq I\otimes X_{k}\xrightarrow{\iota_{a}\otimes \id_{X_{k}}} \coprod_{a\in S}I\otimes X_{k}\simeq X_{0}\otimes X_{k}
$$
and similarly for $c_{*,b}$.

Now note that for the morphisms $c_{a,b}: X_{k+l}\rightarrow X_{0}\otimes X_{k+l}\otimes X_{0}$ and $e_{a,b}: X_{k+l}(a,b)\rightarrow X_{k+l}$, we have
\begin{align*}
(\mu_{0,k,0}\otimes \id_{X_{0}}\otimes \mu_{0,l,0})\mu_{k,0,l}e_{a,b} = (\id_{X_{0}}\otimes \mu_{k,0,0,0,l}\otimes \id_{X_{0}})\mu_{0,k+l,0}e_{a,b}\\
= (\id_{X_{0}}\otimes \mu_{k,0,0,0,l}\otimes \id_{X_{0}})c_{a,b}e_{a,b} = (c_{a,*}\otimes \mu_{0,0,0}\otimes c_{*,b})\mu_{k,0,l}e_{a,b}
\end{align*}
Thus there is a unique $\mu^{a,b}_{k,l}: X_{k+l}(a,b)\rightarrow \coprod_{c\in S}X_{k}(a,c)\otimes X_{l}(c,b)$ such that $(\coprod_{c}e_{a,c}\otimes e_{c,b})\mu^{a,b}_{k,l} = \mu_{k,0,l}e_{a,b}$. Composing this equality with the codiagonal $\coprod_{c}X_{k}\otimes X_{l}\rightarrow X_{k}\otimes X_{l}$, the result follows.
\end{enumerate}
\end{proof}

\begin{Con}
Assume $\mathcal{V}$ is decomposing. We construct a functor
$$
\mathfrak{d}: \Colax_{b}(\fint^{op},\mathcal{V})\rightarrow \ts\mathcal{V}
$$

Take a based colax monoidal functor $X$ of $\mathcal{V}$ with base $S$, comultiplication $\mu$ and counit $\epsilon$. By Construction \ref{construction: decomposing equalizer of cc-functors}, we have a collection of functors $X(a,b): \fint^{op}\rightarrow \mathcal{V}$ for $a,b\in S$, which we can regard as a functor
$$
\tilde{X}: \fint^{op}\rightarrow \mathcal{V}\Quiv_{S}
$$
By Proposition \ref{proposition: decomposition of Fts-object}.$2$, we have a quiver isomorphism $\tilde{\epsilon}: \tilde{X}_{0}\xrightarrow{\sim} I_{S}$, and the morphisms $\mu^{a,b}_{k,l}$ of Proposition \ref{proposition: decomposition of Fts-object}.$3$ combine to give a quiver morphism
$$
\tilde{\mu}_{k,l}: \tilde{X}_{k+l}\rightarrow \tilde{X}_{k}\otimes_{S} \tilde{X}_{l}
$$
It follows from the coassociativity and counitality of $\mu$ and $\epsilon$ that $\tilde{\mu}$ and $\tilde{\epsilon}$ define a strongly unital, colax monoidal structure on $\tilde{X}$ and thus $(\tilde{X},S)$ is a templicial object in $\mathcal{V}$.

Next, let $X$ and $Y$ be based colax monoidal functors of $\mathcal{V}$ with respective bases $S$ and $T$. Let $\alpha: X\rightarrow Y$ be a morphism of based colax monoidal functors. As $\alpha$ is a monoidal natural transformation, there exist unique morphisms $\alpha^{a,b}: X(a,b)\rightarrow Y(f(a),f(b))$ such that $e_{f(a),f(b)}\alpha^{a,b} = \alpha e_{a,b}$, for all $a,b\in S$. This defines a natural transformation $\tilde{X}\rightarrow f^{*}\tilde{Y}$. It further follows from the monoidality of $\alpha$ that the corresponding natural transformation $\tilde{\alpha}: f_{!}\tilde{X}\rightarrow \tilde{Y}$ is monoidal. Hence, $(\tilde{\alpha},f)$ is a morphism of templicial objects $\tilde{X}\rightarrow \tilde{Y}$.

If further $\beta: Y\rightarrow Z$ is a morphism of based colax monoidal functors, then by uniqueness, $(\beta\circ \alpha)^{a,b} = \beta^{f(a),f(b)}\circ \alpha^{a,b}$ for all $a,b\in S$. It follows that the assignments $X\mapsto (\tilde{X},S)$ and $\alpha\mapsto (\tilde{\alpha},f)$ define a functor.
\end{Con}

\begin{Thm}\label{theorem: comparison is equiv.}
Suppose $\mathcal{V}$ is decomposing. Then we have an adjoint equivalence of categories
\[\begin{tikzcd}
	{\ts\mathcal{V}} & {\Colax_{b}(\fint^{op},\mathcal{V})}
	\arrow["{\mathfrak{c}}", shift left=2, from=1-1, to=1-2]
	\arrow["{\mathfrak{d}}", shift left=2, from=1-2, to=1-1]
	\arrow["\sim"{description}, draw=none, from=1-1, to=1-2]
\end{tikzcd}\]
\end{Thm}
\begin{proof}
The isomorphism of Proposition \ref{proposition: decomposition of Fts-object}.$1$ is monoidal by $2.$ and $3.$ of the same Proposition. Moreover, it is directly seen to be natural in $X$. Thus $\mathfrak{c}\circ \mathfrak{d}\simeq \id$.

Let $(X,S)$ be a templicial object of $\mathcal{V}$. We have a functor $X(a,b): \fint^{op}\rightarrow \mathcal{V}$ for every $a,b\in S$. As coproducts are disjoint in $\mathcal{V}$, the equalizer of $\iota_{a,b},\iota_{c,d}: X(c,d)\rightarrow \coprod_{x,y\in S}X(c,d)$ in $\Fun(\fint^{op},\mathcal{V})$ is $X(a,b)$ if $(c,d) = (a,b)$ and $0$ otherwise. Because coproducts commute with decomposing equalizers, we get an equalizer diagram
\[\begin{tikzcd}
	{X(a,b)} & {\underset{c,d\in S}{\coprod}X(c,d)} & {\underset{c,d\in S}{\coprod}\underset{x,y\in S}{\coprod}X(c,d)}
	\arrow["{\iota_{a,b}}", from=1-1, to=1-2]
	\arrow["{\coprod_{c,d}\iota_{c,d}}", shift left=1, from=1-2, to=1-3]
	\arrow["{\coprod_{c,d}\iota_{a,b}}"', shift right=1, from=1-2, to=1-3]
\end{tikzcd}\]
Now $\coprod_{c,d}X(c,d)$ is the functor underlying $\mathfrak{c}(X,S)$ and the morphisms $\coprod_{c,d}\iota_{c,d}$ and $\coprod_{c,d}\iota_{a,b}$ correspond to the induced morphisms $\mu_{0,-,0}$ and $c_{a,b}$ on $\mathfrak{c}(X,S)$ respectively. Consequently, we have an isomorphism between the underlying functors of $(X,S)$ and $\mathfrak{d}\mathfrak{c}(X,S)$. It follows from the definitions that this isomorphism is monoidal and that it is natural in $(X,S)$. Therefore $\mathfrak{d}\circ \mathfrak{c}\simeq \id$.

Finally, the triangle identities are easily verified.
\end{proof}

\section{Templicial objects versus enriched Segal precategories}\label{section: Templicial objects versus enriched Segal precategories}

In this section, we compare templicial objects to the enriched Segal precategories of \cite{lurie2009goodwillie} for cartesian $\mathcal{V}$. Following \cite{simpson2012homotopy}, we consider conditions on $\mathcal{V}$ for which they become equivalent in \S\ref{subsection: The cartesian case}. Next, in \S\ref{subsection: The non-cartesian case} we consider Bacard's generalization of enriched precategories for non-cartesian $\mathcal{V}$ \cite{bacard2010Segal} and show that these cannot reasonably be equivalent to templicial objects in general.

\begin{Def}[\cite{lurie2009goodwillie}, Definition 2.1.1]
Let $S$ be a set. Consider the forgetful functor $\simp\rightarrow \Set$, then we define the comma category:
$$
\simp_{S} = (\simp\downarrow S)
$$
Explicitly, the objects of $\simp_{S}$ may be identified with sequences $(a_{i})_{i=0}^{n} = (a_{0},\dots,a_{n})$ for some $n\geq 0$ and $a_{i}\in S$. A morphism $(a_{i})_{i=0}^{n}\rightarrow (b_{j})_{j=0}^{m}$ in $\simp_{S}$ is given by a morphism $h: [n]\rightarrow [m]$ in $\simp$ such that $b_{h(i)} = a_{i}$ for all $i\in [n]$.

Given a map of sets $f: S\rightarrow T$, let $\simp(f): \simp_{S}\rightarrow \simp_{T}$ denote the induced functor on comma categories. That is, $\simp(f)(a_{0},\dots,a_{n}) = (f(a_{0}),\dots,f(a_{n})))$ for $(a_{i})_{i=0}^{n}\in \simp_{S}$.
\end{Def}

We are interested in functors $X: \simp^{op}_{S}\rightarrow \mathcal{V}$ for some set $S$. In particular, such a functor specifies an object $X(a_{0},\dots,a_{n})\in \mathcal{V}$ for each sequence $(a_{0},\dots,a_{n})$ of elements of $S$. Moreover, for each morphism $h: [m]\rightarrow [n]$ in $\simp$, we have a morphism
$$
X(a_{0},\dots,a_{n})\rightarrow X(a_{h(0)},\dots,a_{h(m)})
$$
in $\mathcal{V}$ which is compatible with compositions and identities in $\simp$.

\begin{Def}[\cite{lurie2009goodwillie}, Definition 2.1.3]\label{definition: enriched Segal precat.}
Assume $(\mathcal{V},\times,1)$ is cartesian monoidal. A \emph{$\mathcal{V}$-enriched precategory} is a pair $(X,S)$ with $S$ a set and $X: \simp^{op}_{S}\rightarrow \mathcal{V}$ a functor such that the terminal map $X(a_{0})\rightarrow 1$ is an isomorphism for all $a_{0}\in S$.

A \emph{morphism} $(X,S)\rightarrow (Y,T)$ of $\mathcal{V}$-enriched precategories is a pair $(\alpha,f)$ with $f: S\rightarrow T$ a map of sets and $\alpha: X\rightarrow Y\circ \simp(f)$ a natural transformation. Composition and identities of such morphisms are defined in the obvious way, and we denote
$$
\mathbf{PC}(\mathcal{V})
$$
for the category of $\mathcal{V}$-enriched precategories and morphisms between them.
\end{Def}

\subsection{The cartesian case}\label{subsection: The cartesian case}

For this subsection, we assume that $(\mathcal{V},\times,1)$ is cartesian monoidal. Following \cite{simpson2012homotopy}, we consider some conditions on $\mathcal{V}$ (Definition \ref{definition: DISJ}) that allow to compare $\mathbf{PC}(\mathcal{V})$ with the category of $S_{b}\mathcal{V}$ of Remark \ref{remark: based colax functors in a cartesian category}. We then show that under these conditions, all the models considered so far coincide.

\begin{Con}
Recall the free functor $F: \Set\rightarrow \mathcal{V}: S\mapsto \coprod_{a\in S}1$. Fix a set $S$ and let $f: A\rightarrow F(S)$ be a morphism in $\mathcal{V}$. Given $a\in S$, we denote $f^{-1}(a)$ for the pullback
$$
\xymatrix{
f^{-1}(a)\ar[r]\ar[d] & A\ar[d]^{f}\\
1\ar[r]_{\iota_{a}} & F(S)
}
$$
The assignment $f\mapsto (f^{-1}(a))_{a\in S}$ extends to a functor $\mathfrak{d}_{S}: \mathcal{V}/F(S)\rightarrow \prod_{a\in S}\mathcal{V}$.

Conversely, let $(A_{a})_{a\in S}$ be a collection of objects in $\mathcal{V}$. Then we obtain a morphism $\coprod_{a\in S}A_{a}\rightarrow \coprod_{a\in S}1\simeq F(S)$. This extends to a functor $\mathfrak{c}_{S}: \prod_{a\in S}\mathcal{V}\rightarrow \mathcal{V}/F(S)$.

Note that to any morphism $f: A\rightarrow F(S)$, we can associate a canonical morphism $\coprod_{a\in S}f^{-1}(a)\rightarrow A$ over $F(S)$, and for any collection $(A_{a})_{a\in S}$, there are canonical morphisms $A_{a}\rightarrow (\coprod_{a\in S}t)^{-1}(a) = \mathfrak{d}_{S}\mathfrak{c}_{S}((A_{a})_{a\in S})$ in $\mathcal{V}$ for each $a\in S$. These form the counit and unit of an adjunction
\[\begin{tikzcd}
	{\prod_{a\in S}\mathcal{V}} & {\mathcal{V}/F(S)}
	\arrow[""{name=0, anchor=center, inner sep=0}, "{\mathfrak{c}_{S}}", shift left=2, from=1-1, to=1-2]
	\arrow[""{name=1, anchor=center, inner sep=0}, "{\mathfrak{d}_{S}}", shift left=2, from=1-2, to=1-1]
	\arrow["\dashv"{anchor=center, rotate=-90}, draw=none, from=0, to=1]
\end{tikzcd}\] 
\end{Con}

\begin{Def}[\cite{simpson2012homotopy}, Condition 10.7.1]\label{definition: DISJ}
We say $\mathcal{V}$ \emph{satisfies (DISJ)} if
\begin{enumerate}[(a)]
\item For any set $S$, the adjunction $\mathfrak{c}_{S}\dashv \mathfrak{d}_{S}$ is an equivalence of categories.
\item The terminal object $1$ in $\mathcal{V}$ is indecomposable. That is, for all $A,B\in \mathcal{V}$, $1\simeq A\amalg B$ implies that either $A$ or $B$ is an initial object of $\mathcal{V}$.
\item The category $\mathcal{V}$ has at least two non-isomorphic objects.
\end{enumerate}
\end{Def}

It is shown in \cite[Lemma 10.7.2(5)]{simpson2012homotopy} that the free functor $F: \Set\rightarrow \mathcal{V}$ is fully faithful when $\mathcal{V}$ satisfies (DISJ). Thus we may view $\Set$ as a full subcategory of $\mathcal{V}$.  It is then clear from the definition of $S_{b}\mathcal{V}$ as a $2$-pullback \eqref{diagram: based simp. objs.} that it is equivalent to the full subcategory of $S\mathcal{V}$ spanned by those simplicial objects $X: \simp^{op}\rightarrow \mathcal{V}$ with $X_{0}\in \Set$. We can then reformulate a result of Simpson as follows.

\begin{Thm}[\cite{simpson2012homotopy}, Theorem 10.7.3]\label{theorem: simpson}
Assume that $\mathcal{V}$ satisfies (DISJ). Then we have an equivalence of categories
$$
\mathbf{PC}(\mathcal{V})\simeq S_{b}\mathcal{V}
$$
\end{Thm}

\begin{Lem}\label{lemma: coproducts commute with equalizers and pullbacks under DISJ}
Suppose that $\mathcal{V}$ satisfies condition $(a)$ of (DISJ). Then coproducts commute with equalizers and pullbacks in $\mathcal{V}$.
\end{Lem}
\begin{proof}
Under the equivalence $\prod_{a\in S}\mathcal{V}\simeq \mathcal{V}/F(S)$, the forgetful functor $\mathcal{V}/F(S)\rightarrow \mathcal{V}$ coincides with the coproduct functor $\prod_{a\in S}\mathcal{V}\rightarrow \mathcal{V}: (A_{a})_{a\in S}\mapsto \coprod_{a\in S}A_{a}$. The result follows by noting that $\mathcal{V}/F(S)\rightarrow \mathcal{V}$ preserves pullbacks and equalizers.
\end{proof}

\begin{Prop}\label{proposition: DISJ implies decomposing}
Assume that $\mathcal{V}$ satisfies (DISJ). Then $\mathcal{V}$ is decomposing.
\end{Prop}
\begin{proof}
Since $\mathcal{V}$ is cartesian, the monoidal product preserves any equalizer in each variable. Moreover, it is shown in \cite[Lemma 10.7.2(7)]{simpson2012homotopy} that coproducts are disjoint in $\mathcal{V}$. Finally, it follows from Lemma \ref{lemma: coproducts commute with equalizers and pullbacks under DISJ} that coproducts commute with all equalizers.
\end{proof}

The converse to Proposition \ref{proposition: DISJ implies decomposing} does not hold.

\begin{Exs}
\begin{itemize}
\item Consider the cartesian monoidal category $k\Coalg$ of coalgebras over a field $k$, which is decomposing by Proposition \ref{proposition: stability of decomposing under constructions}.2. However, $k\Coalg$ does not satisfy $(a)$ of (DISJ). Indeed, consider the cosimple coalgebra of $(n\times n)$-matrices $M_{n}(k)$ with comultiplication $\Delta$ and counit $\epsilon$ given on elementary matrices by $\Delta(E_{ij}) = \sum_{k=1}^{n}E_{ik}\otimes E_{kj}$ and $\epsilon(E_{ij}) = \delta_{ij}$ (the Kronecker delta).

Let us take $n = 2$. It is easy to check that we have a coalgebra map
$$
f: M_{2}(k)\rightarrow k\oplus k: (a_{pq})_{pq}\mapsto (a_{11},a_{22})
$$
Consider the pullbacks $f^{-1}(1), f^{-1}(2)$ of $f$ along the canonical inclusions $\iota_{1},\iota_{2}: k\hookrightarrow k\oplus k$. Then $f^{-1}(i)$ is a subcoalgebra of $M_{2}(k)$ that only contains matrices $(a_{pq})_{pq}$ with $a_{ii} = 0$. Since $M_{2}(k)$ is cosimple, this implies that $f^{-1}(i) = \{0\}$. But then clearly $f^{-1}(1)\oplus f^{-1}(2)$ does not recover $M_{2}(k)$.
\item Consider the monoid $(\{0,1\},*,1)$ in $\Set$ with $0*0 = 0*1 = 1*0 = 0$ and $1*1 = 1$. Then by Proposition \ref{proposition: stability of decomposing under constructions}.1, the over category $\Set/\{0,1\}$ is still decomposing, but its terminal object (the identity on) $\{0,1\}$ is not indecomposable since $\{0\}\amalg \{1\}\simeq \{0,1\}$. So $\Set/\{0,1\}$ does not satisfy $(b)$ of (DISJ).
\item The category $[0]$ with a single object and morphism, with the unique monoidal structure. It is trivially decomposing but doesn't satisfy $(c)$ of (DISJ).
\end{itemize}
\end{Exs}

\begin{Cor}\label{corollary: temp. objs. vs precats.}
Assume that $\mathcal{V}$ satisfies (DISJ). Then we have equivalences
$$
\ts\mathcal{V}\simeq \Colax_{b}(\fint^{op},\mathcal{V})\simeq S_{b}\mathcal{V} \simeq \mathbf{PC}(\mathcal{V})
$$
\end{Cor}
\begin{proof}
Combine Remark \ref{remark: based colax functors in a cartesian category}, Theorems \ref{theorem: comparison is equiv.} and \ref{theorem: simpson}, and Proposition \ref{proposition: DISJ implies decomposing}.
\end{proof}

\subsection{The non-cartesian case}\label{subsection: The non-cartesian case}

In \cite{bacard2010Segal}, Bacard introduced precategories enriched in a general (that is, not necessarily cartesian) monoidal category $\mathcal{V}$ as pairs $(X,S)$ with $S$ a set and $X$ a strongly unital, colax functor of bicategories
$$
X: \mathcal{P}_{\overline{S}}\rightarrow \mathcal{V}
$$
where $\mathcal{V}$ is considered as a bicategory with one object and $\mathcal{P}_{\overline{S}}$ is a certain $2$-category, which is a labeled version of $\fint^{op}$ analogously to how $\simp_{S}$ is a labeled version of $\simp$. For more details we defer to loc. cit., where $X$ is referred to as a \emph{unital $\overline{S}$-point}.

In particular, such an enriched precategory $(X,S)$ specifies objects $X(a_{0},\dots,a_{n})\in \mathcal{V}$ for each sequence $(a_{0},\dots,a_{n})$ of elements of $S$. Moreover, for each morphism $h: [m]\rightarrow [n]$ in $\fint$, we have a morphism
$$
X(a_{0},\dots,a_{n})\rightarrow X(a_{h(0)},\dots,a_{h(m)})
$$
in $\mathcal{V}$, as well as comultiplication maps (for all $k,l\geq 0$) and counit isomorphisms
$$
X(a_{0},\dots,a_{k+l})\rightarrow X(a_{0},\dots,a_{k})\otimes X(a_{k},\dots,a_{k+l})\quad \text{and}\quad X(a_{0})\xrightarrow{\sim} I
$$
satisfying coassociativity, counitality and naturality conditions.

Let us denote the category of $\mathcal{V}$-enriched precategories in this sense by $\mathbf{PC}(\mathcal{V})$ as well. If $\mathcal{V}$ is cartesian, this recovers the enriched precategories of \cite{lurie2009goodwillie}\cite{simpson2012homotopy} so there is no risk of confusion. Then, similarly to Construction \ref{construction: comparison functor}, we can construct a comparison functor
$$
\mathfrak{c}': \mathbf{PC}(\mathcal{V})\rightarrow \ts\mathcal{V}
$$
which is defined on objects as follows. Given a $\mathcal{V}$-enriched precategory $(X,S)$, we can define a templicial object $(\mathfrak{c}(X),S)$ by setting for all $a,b\in S$ and $n > 0$:
$$
\mathfrak{c}'(X)_{n}(a,b) = \coprod_{a_{1},\dots,a_{n-1}\in S}X(a,a_{1}\dots,a_{n-1},b)
$$
This clearly extends to a functor $\fint^{op}\rightarrow \mathcal{V}\Quiv_{S}$. The colax monoidal structure on $\mathfrak{c'}(X)$ is induced by the comultiplication maps of $X$, and the strong unitality by the counit isomorphisms of $X$.

In general, there is no hope for this comparison functor to be an equivalence however, even in cases of interest such as $\mathcal{V} = \Mod(k)$.

\begin{Ex}\label{example: non-cartesian temp. obj. vs precats}
Let $\mathcal{V} = \Mod(k)$ be the category of $k$-modules for a commutative unital ring $k$. Note that $\Mod(k)$ is decomposing by Proposition \ref{proposition: preadditive implies decomposing}. But
$$
\mathfrak{c}': \mathbf{PC}(\Mod(k))\rightarrow \ts\Mod(k)
$$
is not essentially surjective.

Indeed, let $Y$ be the templicial object of \cite[Example 2.10]{lowen2024enriched}. It has vertices $S = \{a,c_{0},c_{1},b\}$ and the modules $Y_{1}(x,y)$ for $x\neq y$ are generated by edges $f_{i}: a\rightarrow c_{i}$, $g_{i}: c_{i}\rightarrow b$ and $h: a\rightarrow b$. Further, it has $2$-simplex $w\in Y_{2}(a,b)$ satisfying
$$
d_{1}(w) = h\quad \text{and}\quad \mu_{1,1}(w) = f_{1}\otimes g_{1} + f_{2}\otimes g_{2}
$$
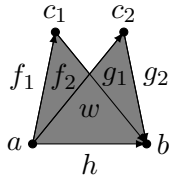
\begin{center}
\begin{tikzpicture}[scale=1.5]
\filldraw[fill=gray,opacity=0.3]
(0,-0.289) -- (0.2,0.7) -- (1,-0.289);
\filldraw[fill=gray,opacity=0.3]
(0,-0.289) -- (0.8,0.7) -- (1,-0.289);
\filldraw
(0,-0.289) circle (1pt) node[left]{$a$}
(0.2,0.7) circle (1pt) node[above]{$c_{1}$}
(0.8,0.7) circle (1pt) node[above]{$c_{2}$}
(1,-0.289) circle (1pt) node[right]{$b$}
(0.5,0) node{$w$};
\draw[-latex] (0,-0.289) -- node[left,pos=0.6]{$f_{1}$} (0.2,0.7);
\draw[-latex] (0,-0.289) -- node[left,pos=0.6]{$f_{2}$} (0.8,0.7);
\draw[-latex] (0.2,0.7) -- node[right,pos=0.4]{$g_{1}$} (1,-0.289);
\draw[-latex] (0.8,0.7) -- node[right,pos=0.4]{$g_{2}$} (1,-0.289);
\draw[-latex] (0,-0.289) -- node[below,pos=0.5]{$h$} (1,-0.289);
\end{tikzpicture}
\end{center}

Suppose there exists some $X\in \mathbf{PC}(\Mod(k))$, such that $\mathfrak{c}_{inner}(X)\simeq Y$. Then we would have an induced isomorphism in $\Mod(k)$:
$$
X(a,a,b)\oplus X(a,c_{0},b)\oplus X(a,c_{1},b)\oplus X(a,b,b)\simeq Y_{2}(a,b)
$$
And thus we can write $w = \sum_{x\in S}w_{x}$ with $\mu_{1,1}(w_{x})\in Y_{1}(a,x)\otimes Y_{1}(x,b)$. But no such $2$-simplices $w_{x}$ exist in $Y$.
\end{Ex}

\begin{Ex}\label{example: failure of dg-nerve for precats}
Let $k$ be a field and consider the templicial dg-nerve \cite{lowen2023frobenius}
$$
N^{dg}_{k}: k\Cat_{dg}\rightarrow \ts\Mod(k)
$$
from the category of small (homologically graded)  dg-categories to templicial $k$-modules. Then there exists no functor $N': k\Cat_{dg}\rightarrow \mathbf{PC}(\Mod(k))$ so that we have isomorphisms $N^{dg}_{k}\simeq \mathfrak{c}'\circ N'$ and $N'(\mathcal{C}^{op})\simeq N'(\mathcal{C})^{op}$ for any dg-category $\mathcal{C}$.

Indeed, consider the dg-category $\mathcal{C}_{\bullet}$ with object set $S = \{a,b\}$ and morphisms given by the following complexes:
$$
\mathcal{C}(a,a) = k\id_{a},\qquad \mathcal{C}(b,b) = k\id_{b},\qquad \mathcal{C}(a,b) = kw,\qquad \mathcal{C}(b,a) = 0
$$
with $\id_{a}$ and $\id_{b}$ of degree $0$ and $w$ of degree $1$. Then it follows from \cite[Example 3.19]{lowen2023frobenius} that the $2$-simplices from $a$ to $b$ of $N^{dg}_{k}(\mathcal{C})$ are given by 
$$
N^{dg}_{k}(\mathcal{C})_{2}(a,b)\simeq kw
$$
The key idea is that the $2$-simplex $w$ doesn't have a canonical inner vertex.

Note that we have an isomorphism $\varphi: \mathcal{C}^{op}\xrightarrow{\sim} \mathcal{C}$ which sends $a\mapsto b$, $b\mapsto a$ and $w\mapsto w$. So for any functor $N'$ as above, we would obtain induced isomorphisms
$$
N^{dg}_{k}(\mathcal{C})_{2}(a,b)\simeq N'(\mathcal{C})(a,a,b)\oplus N'(\mathcal{C})(a,b,b)\simeq N'(\mathcal{C})(a,a,b)^{2}
$$
this is impossible as seen by comparing dimensions.
\end{Ex}

\printbibliography

@phdthesis{aguiar1997internal,
    AUTHOR = {Aguiar, Marcelo},
     TITLE = {Internal categories and quantum groups},
 PUBLISHER = {ProQuest LLC, Ann Arbor, MI},
    SCHOOL = {Cornell University},
      YEAR = {1997},
     PAGES = {295},
  MRNUMBER = {2696373}
}

@Online{bacard2010Segal,
  hyphenation = {american},
  author      = {Bacard, Hugo V.},
  title       = {Segal enriched categories {I}},
  version     = {1},
  date        = {2010-09-20},
  eprinttype  = {arxiv},
  eprint      = {1009.3673v1},
  primaryClass= {math.CT}
}

@Online{hirschowitz1998descente,
  hyphenation = {american},
  author      = {Hirschowitz, André and Simpson, Carlos},
  title       = {Descente pour les n-champs (Descent for n-stacks)},
  version     = {3},
  date        = {2001-03-13},
  eprinttype  = {arxiv},
  eprint      = {math/9807049v3},
  primaryClass= {math.AG}
}

@article {joyal2002quasi,
    AUTHOR = {Joyal, Andr\'{e}},
     TITLE = {Quasi-categories and {K}an complexes},
      NOTE = {Special volume celebrating the 70th birthday of Professor Max
              Kelly},
   JOURNAL = {J. Pure Appl. Algebra},
  FJOURNAL = {Journal of Pure and Applied Algebra},
    VOLUME = {175},
      YEAR = {2002},
    NUMBER = {1-3},
     PAGES = {207--222},
   MRCLASS = {55U10 (18G55)},
  MRNUMBER = {1935979},
MRREVIEWER = {Donald M. Davis},
       %DOI = {10.1016/S0022-4049(02)00135-4},
}

@Online{leinster2000homotopy,
  hyphenation = {american},
  author      = {Leinster, Tom},
  title       = {Homotopy algebras for operads},
  version     = {1},
  date        = {2000-02-22},
  eprinttype  = {arxiv},
  eprint      = {math/0002180v1},
  primaryClass= {math.QA}
}

@article {lowen2024enriched,
    AUTHOR = {Wendy Lowen and Arne Mertens},
     TITLE = {Enriched quasi-categories and the templicial homotopy coherent nerve},
   JOURNAL = {Alg. Geom. Topol.},
      YEAR = {2024},
eprinttype = {arxiv},
	eprint = {2302.02484v2},
}

@online{lowen2023frobenius,
  hyphenation = {american},
  author      = {Wendy Lowen and Arne Mertens},
  title       = {Frobenius templicial modules and the dg-nerve},
  version     = {3},
  date        = {2023-05-18},
  eprinttype  = {arxiv},
  eprint      = {2005.04778v3},
  primaryClass= {math.CT}
}

@Online{lurie2009goodwillie,
  hyphenation = {american},
  author      = {Lurie, Jacob},
  title       = {(Infinity,2)-Categories and the Goodwillie Calculus I},
  version     = {2},
  date        = {2009-05-08},
  eprinttype  = {arxiv},
  eprint      = {0905.0462v2},
  primaryClass= {math.CT}
}

@phdthesis{mertens2022templicial,
    AUTHOR = {Mertens, Arne},
     TITLE = {Templicial objects: simplicial objects in a monoidal category},
 PUBLISHER = {},
    SCHOOL = {University of Antwerp},
      YEAR = {2022},
     PAGES = {172},
  MRNUMBER = {}
}

@book {simpson2012homotopy,
    AUTHOR = {Simpson, Carlos},
     TITLE = {Homotopy theory of higher categories},
    SERIES = {New Mathematical Monographs},
    VOLUME = {19},
 PUBLISHER = {Cambridge University Press, Cambridge},
      YEAR = {2012},
      %DOI  = {10.1017/CBO9780511978111}
}

@inproceedings {street1974elementary,
    AUTHOR = {Street, Ross},
     TITLE = {Elementary cosmoi. {I}},
 BOOKTITLE = {Category {S}eminar ({P}roc. {S}em., {S}ydney, 1972/1973)},
    SERIES = {Lecture Notes in Math., Vol. 420},
     PAGES = {134--180},
 PUBLISHER = {Springer, Berlin},
      YEAR = {1974},
   MRCLASS = {18D99},
  MRNUMBER = {0354813},
MRREVIEWER = {H. Gonshor},
}

\end{document}